\documentclass{amsart}
\usepackage[latin1]{inputenc}
\usepackage{amssymb,amsfonts,amsthm}
\usepackage{mathrsfs}
\usepackage{graphicx}
\usepackage{hyperref}
\usepackage{bbm}
\usepackage{xcolor}
\usepackage{tabularx}

\newtheorem{theorem}{Theorem}
\newtheorem{corollary}[theorem]{Corollary}
\newtheorem{definition}[theorem]{Definition}

\newtheorem{lemma}[theorem]{Lemma}
\newtheorem{proposition}[theorem]{Proposition}
\newtheorem{remark}[theorem]{Remark}

\newcommand{\R}{{\mathbb R}}

\newcommand{\Ff}{{\mathcal F}}
\newcommand{\cS}{{\mathcal S}}
\newcommand{\cB}{{\mathcal B}}
\newcommand{\cN}{{\mathcal N}}
\newcommand{\cL}{{\mathcal L}}

\newcommand{\E}{{\mathbb E}}

\newcommand{\eps}{\epsilon}

\newcommand{\su}{\underline{\sigma}}
\newcommand{\ds}{\displaystyle}
\newcommand{\Eh}{\hat{\mathbb{E}}}

\begin{document}

\title[Model reduction of parametric multiscale diffusions]{Model reduction and uncertainty quantification of multiscale diffusions with parameter uncertainties using nonlinear expectations}

\author{Hafida Bouanani}

\address{Laboratory of Stochastic Models, Statistics and Applications, University of Saida
	Dr Moulay Tahar, Algeria}

\email{hafida.bouanani@univ-saida.dz}

\author{Carsten Hartmann \and Omar Kebiri}

\address{Institute of Mathematics, Brandenburgische Technische Universit\"at Cottbus-Senftenberg, 03046 Cottbus, Germany}

\email{carsten.hartmann@b-tu.de, omar.kebiri@b-tu.de}

\begin{abstract}
	In this paper we study model reduction of linear and bilinear quadratic stochastic control problems with parameter uncertainties. Specifically, we consider slow-fast systems with unknown diffusion coefficient and study the convergence of the slow process in the limit of infinite scale separation. The aim of our work is two-fold: Firstly, we want to propose a general framework for averaging and homogenisation of multiscale systems with parametric uncertainties in the drift or in the diffusion coefficient. Secondly, we want to use this framework to quantify the uncertainty in the reduced system by deriving a limit equation that represents a worst-case scenario for any given (possibly path-dependent) quantity of interest. We do so by reformulating the slow-fast system as an optimal control problem in which the unknown parameter plays the role of a control variable that can take values in a closed bounded set. For systems with unknown diffusion coefficient, the underlying stochastic control problem admits an interpretation in terms of a stochastic differential equation driven by a G-Brownian motion. We prove convergence of the slow process with respect to the nonlinear expectation on the probability space induced by the G-Brownian motion. The idea here is to formulate the nonlinear dynamic programming equation of the underlying control problem as a forward-backward stochastic differential equation in the G-Brownian motion framework (in brief: G-FBSDE), for which convergence can be proved by standard means. We illustrate the theoretical findings with two simple numerical examples, exploiting the connection between fully nonlinear dynamic programming equations and second-order BSDE (2BSDE): a linear quadratic Gaussian regulator problem and a bilinear multiplicative triad that is a standard benchmark system in turbulence and climate modelling.
\end{abstract}


\maketitle

\section{Introduction}

Modelling of real-world processes often involves high-dimensional stochastic differential equations with multiple time and length scales. In many cases the relevant behaviour is given by the largest scales in the system (e.g., phase transitions), the direct numerical simulation of which is difficult. Moreover, the system variables may be only be partially observable, or the model may involve unknown parameters, which makes the numerical or analytical treatment of the high-dimensional equations by multiscale techniques as well as the uncertainty quantification (UQ) of derived quantities difficult. Recently, data-driven approaches have been developed in order to account for model uncertainties or partial observability \cite{pmor}. The idea there is to postulate a reduced-order model for some given quantities of interest (QoI), based on either physical principles \cite{generic} or classical multiscale techniques such as averaging and homogenisation \cite{mtv}, and then parameterise the resulting model equations using dense observations of the resolved variables or the QoI \cite{hmm}. For multiscale systems with more than two scales, the parameter estimation is known to be a difficult task, for the standard estimators may be strongly biased and require sophisticated subsampling strategies to reduce the bias \cite{mle}.

\subsection*{A sublinear expectation framework for multiscale diffusions}

We study slow-fast stochastic differential equations (SDE) with unobservable fast variables where the latter are driven by a G-Brownian motion (G-SDE). The approach pursued in this paper is different from the aforementioned methods in that it employs analytical techniques for the elimination of fast variables, like averaging or homogenisation that normally require that the model is fully specified (with all parameters being available), and yet takes the inherent uncertainty of the unresolved or partially resolved (``underresolved'') fast variables into account. Specifically, the parameter uncertainty of the underresolved fast variables is taken into account by modelling them as a stationary G-SDE that then generates a parametric family of probability distributions that, consequently, give rise to a family of reduced-order models for the slow variables or QoI. More specifically, we will study averaging of G-SDE and forward-backward G-SDE (G-FBSDE) and discuss the relation to traditional SDE and FBSDE multiscale methods (e.g.~\cite{mle,dupuis}).

The idea is to consider path functionals of the resolved variables and to maximize the expected difference between the original and the limit dynamics over the uncertain parameters. Exploiting the specific properties of the driving G-Brownian motion and related nonlinear expectation concepts \cite{Peng07}, we show that the maximum difference goes to zero in the limit of infinite scale separation. Depending on the specific situation at hand (i.e. whether drift or diffusion coefficients are uncertain), the nonlinear expectation boils down to a g-expectation \cite{Pengg} and an averaging problem for a standard (uncoupled) FBSDE, or to a G-expectation \cite{Peng07} and an averaging problem for a G-FBSDE.
For finite scale separation, the maximiser of the functional generates a worst-case scenario for the deviation between  the multiscale and the limit dynamics, and thus the maximiser quantifies the uncertainty in the QoI. For finite scale separation, the maximiser of the functional generates a worst-case scenario for the deviation between  the multiscale and the limit dynamics, and thus the maximiser quantifies the uncertainty in the QoI.

\subsection*{Existing work}

Backward stochastic differential equations (BSDE) as a probabilistic representation of semilinear partial differential equations (PDE) have received a lot of attention, starting with the work of Peng and Pardoux \cite{PP90}. After that, the theory of forward-backward stochastic differential equations (FBSDE) and their applications to stochastic control problems developed quickly, following the work of Antonili \cite{A93}; see also \cite{Ma,mayou}.
Recently, Redjil et al. \cite{RC} proved the existence of an optimal control of a controlled forward SDE driven by a G-Brownian motion (G-SDE), allowing to model situations in which the randomness in a model coming from, e.g., measurements or model uncertainties does not satisfy the usual i.i.d. assumption, so that the classical limit theorems like the law of large numbers or the central limit theorem do not apply.
The theory and the stochastic calculus for G-SDE have been developed by Peng and co-workers \cite{Peng07,denis2011}. Relevant preliminary work on existence and uniqueness of fully coupled FBSDE, G-FBSDE and the corresponding dynamic programming (Hamilton-Jacobi Bellman or HJB) equations is due to Redjil \& Choutri \cite{RC} and Kebiri et al. \cite{BKM,BKMM,BGM11,BK}, showing the existence of a relaxed control based on results of El-Karoui et al.~\cite{KNJ}.

Systematic model reduction and uncertainty quantification methods for multiscale parametric systems are still at their infancies, notwithstanding recent advances in the field; see~\cite{benner2015,majda2018,mcdowell2020} and the references therein. Related work on model reduction of controlled multiscale diffusions using duality arguments and Fleming's technique of logarithmic transformations (see~\cite[Ch.~VI]{fleming2006}) has been carried by one of the authors \cite{JCD20,NONL16,JCD14,PTRF18}. Despite recent progress on the theoretical foundations of G-(F)BSDE and G-Brownian motion, there have been relatively few practically oriented works in the context of uncertainty quantification (see e.g. \cite{HSP,Peng04,Peng19}) and even fewer on numerical methods for G-Brownian motions (see e.g. \cite{JZ}).

\subsection*{Outline of the article}

The idea of using the G-Brownian motion framework to do uncertainty quantification for multiscale systems is explained in Section \ref{sec:slowfast}. The key theoretical result of this paper, the convergence of the value function and its derivative, is formulated and proved in Section \ref{sec:convergence}. To illustrate the theoretical findings, we discuss two numerical examples with unceetain diffusions in Section \ref{sec:num}: a linear quadratic Gaussian regulator with uncertain diffusion and an uncontrolled bilinear benchmark system from turbulence modelling. We summarise the key observations and main results in Section \ref{sec:sum}.
The article contains two appendices: Appendix \ref{sec:prelim} records basic definitions related to Peng's nonlinear expectation and inequalities for G-Brownian motion that will be used throughout the article. Appendix \ref{sec:2BSDE} records basic definitions and identities related to the stochastic representations of fully nonlinear partial differential equations in terms of second-order BSDE that are used to carry out the numerical simulations in Section \ref{sec:num}

\section{Slow-fast system}\label{sec:slowfast}

Let $x=(r,u)\in\R^n=\R^{n_s}\times\R^{n_f}$ and $\eps>0$ be a small parameter. We consider slow-fast multiscale SDE models of the form
\begin{subequations}
	\begin{align}\label{homx}
		dR^\eps_t & = \left(f_0(R_t^\eps,U_t^\eps) + \frac{1}{\sqrt{\eps}}f_1(R_t^\eps,U_t^\eps)\right)dt + \alpha(R_t^\eps,U_t^\eps)dV_t\\\label{homy}
		dU^\eps_t & = \frac{1}{\eps}g(R_t^\eps,U_t^\eps;\theta)dt + \frac{1}{\sqrt{\eps}}\beta(R_t^\eps,U_t^\eps;\theta)dW_t\,,
	\end{align}
\end{subequations}
where all coefficients are assumed to be such that the SDE has a unique strong solution for all times. We call $R_t^\eps$ the resolved (slow) variable and $U_t^\eps$ the unresolved (fast) variable that is not fully accessible and depends on an unknown parameter $\theta\in\Theta\subset\R^p$, where for convenience we suppress the dependence on $\theta$.

The aim is to derive a closed equation for $R^{\eps}$ for $\eps\to 0$ that best approximates the resolved process whenever $\eps$ is sufficiently small. Since the fast process depends on an unknown parameter, the answer to the question what the \emph{best approximation} is remains ambigous.

\subsection{Goal-oriented uncertainty quantification}

To illustrate the ambiguity in the reduced dynamics, let us consider the degenerate diffusion
\begin{subequations}\label{avgex}
	\begin{align}\label{avgxex}
		dR_t & = (R_{t} - U_t^{3})dt\,,\quad R_{0}=r\\\label{avgyex}
		dU_t & = \frac{1}{\eps}(R_t^\eps - U_t)dt + \sqrt{\frac{2\theta}{\eps}} dW_t\,,\quad U_{0}=u\,.
	\end{align}
\end{subequations}
for $\theta\in[0,1]$ where, for simplicity, we use the shorthand $(R,U)=(R^{\eps},U^{\eps})\in\R\times\R$ and suppress the dependence on the small parameter $\eps$.

When $\eps\ll 1$, the fast dynamics becomes ``slaved'' by the slow dynamics and randomly fluctuates around $R_{t}$. The unique limiting invariant measure of the fast variables conditional on $R_{t}=r$ is given by $\mu_{r}=\cN(r,\theta)$ when $\theta\in(0,1]$, and singular, $\mu_{r}=\delta_{r}$ for $\theta=0$. As $\eps\to 0$ it follows from the averaging principle (e.g.~\cite[Ch.~7]{freidlin}), that the slow process $R=R^{\eps}$ converges pathwise to a limit process that is the solution of the (here: deterministic) initial value problem
\begin{equation}
	\frac{dr}{dt} = F(r;\theta)\,,\quad r(0)=r\,,
\end{equation}
where
\begin{equation}
	F(r,\theta) = - r^{3} + r(1- 3\theta)\,,\quad \theta\in[0,1]\,.
\end{equation}
Figure \ref{fig:ode} shows the vector field $F(\cdot,\theta)$ for three different values of $\theta$ and illustrates that the limit dynamics undergoes a supercritical pitchfork bifurcation at $\theta=1/3$ at which two asymptotically stable fixed point and an unstable one collapse into one asymptotically stable one.
Note that $F(r,\cdot)$ is continuous at $\theta=0$, nevertheless, depending on value of $\theta$,
the qualitative properties of the limit dynamics change drastically as $\theta$ varies.
It therefore makes sense to modify the best approximation question slightly and instead ask for a worst-case scenario in terms of the unknown parameter for a given quantity of interest (QoI).

\begin{figure}
	\begin{center}
		\includegraphics[width=0.65\textwidth]{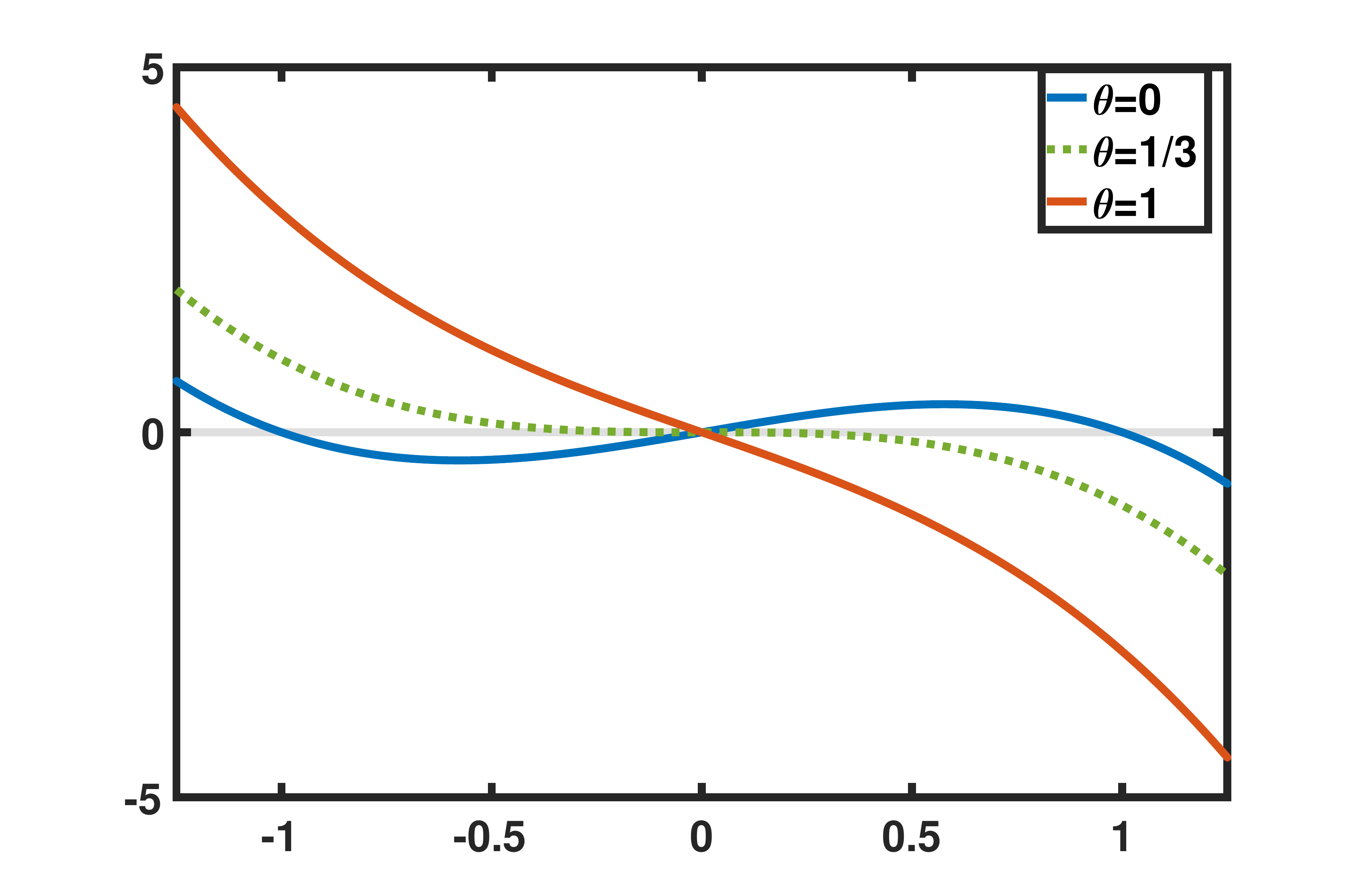}
		\caption{Limiting vector field $F(\cdot,\theta)$ for $\theta\in\{0,\,1/3,\,1\}$. The value $\theta=1/3$ (green dotted curve) corresponds to a supercritical pitchfork bifurcation of the dynamics.}
		\label{fig:ode}
	\end{center}
\end{figure}

Let $\varphi\colon C([0,T])\to\R$ be a suitable test function. The objects of interest are path functionals of the form $\phi^{\eps} = \varphi(R^{\eps})$, with $R^{\eps}=(R^{\eps,\theta}_{t})_{t\in[0,T]}$. To this end let $r=(r_{t}^{\theta})_{t\in[0,T]}$ denote the candidate limit process as $\eps\to 0$ and denote $\phi=\varphi(r)$.

A worst-case scenario for the convergence of $R^{\eps}$ to the limiting process $r$  can be expressed by the G-expectation using the representation formula (\ref{Gexpec}):
\begin{equation}\label{worst-case}
	\hat{\E}\left(|\phi^{\eps}-\phi|\right) = \sup_{\theta\in\Theta}\E_{\theta}\left(|\phi^{\eps}-\phi|\right)\,.
\end{equation}
For example the worst-case approximation for the variance (or the second moment) may be different from the approximation of the slow process itself, in that they correspond to different values of the unknown parameter $\theta$.

If the linear expectation on the right hand side of (\ref{worst-case}) converges for every fixed $\theta\in\Theta$, stability results (e.g.~\cite[Thm.~3.1]{zhang11}) for G-BSDE imply that
\begin{equation}
	\lim_{\eps\to 0}\Eh\left(|\phi^{\eps}-\phi|\right) = 0\,.
\end{equation}
If $\phi^{\eps}$ is regarded as \emph{data}, then the G-expectation defines some kind of tracking problem for the limit dynamics, with $\theta$ playing the role of the control variable. (There may be an additional control variable in the equations though.) An equivalent statement is that the value function, i.e.~the unique viscosity solution of the underlying dynamic programming equation converges as $\eps\to 0$.

One of the messages of the previous considerations is that robust approximations of a multiscale diffusion with parameter uncertainties may depend on the class of test functions $\varphi$ via the optimal parameter $\theta^{*}$. In general, by the dynamic programming principle, $\theta^*=\theta^*(t)$ will be time dependent or a feedback law, therefore the limit equations are not simply obtained by setting $\theta$ equal to some appropriate value. They are moreover goal-oriented, in that they depend on the QoI.

\section{Convergence of the quantity of interest}\label{sec:convergence}

In this section we study the convergence of the slow component of a slow-fast system driven by a G-Brownian motion. Specicifally, we prove convergence of the corresponding value function that is associated with the QoI. For the sake of simplicity, the proof will be given for a linear controlled  G-SDE only, but we stress that the proof carries over to the case of a nonlinear G-SDE with or without control and under standard Lipschitz conditions, using essentially the same techniques.

\subsection*{Controlled linear-quadratic slow-fast system and related QoI}

We consider the following controlled stochastic differential equation
      \begin{equation}\label{xeps}
           dX_{s}^{\eps}=(A^{\epsilon} X_{s}^{\eps}+B^{\epsilon}\alpha_{s})ds+C^{\epsilon}dW_{s}; \quad X_{t}^{\eps}=x,
      \end{equation}
with $X_{s}^{\eps}=(U_s, R_s)$ taking values in $\R^{n_{s}}\times\R^{n_{f}}$ where $n_{s}+n_{f}=n$, and $x=(r,u)$ denotes  the decomposition of the state vector $x$ into slow (resolved) and fast (unresolved) components. We will suppress the dependence of $R$ and $U$ on $\eps$, until further notice. Here $W=(W_{t})_{t\geq0}$ is a standard $\R^{m}$-valued Brownian motion on $(\Omega, \Ff, \mathbb{P})$ that is endowed with its own filtration $(\Ff_{t})_{t\geq0}$, and $\alpha=(\alpha_t)_{t
\ge 0}$ denotes an adapted control variable with values in $\R^k$.
Let
$$A^{\eps}=\left(
               \begin{array}{cc}
                 A_{11} & \eps^{\scriptscriptstyle-1/2}A_{12} \\
                 \eps^{\scriptscriptstyle-1/2}A_{21} & \eps^{-1}A_{22} \\
               \end{array}
             \right)\in\R^{n\times n},$$
with the natural partitioning into $A_{11}\in\R^{n_s\times n_s}$, etc. where assume that the matrix $A_{22}\in\R^{n_f\times n_f}$ is Hurwitz, i.e.~all of its eigenvalues are lying in the open left half-plane. The control and the noise coefficients are partitioned as follows:
$$B^{\epsilon}=\left(
      \begin{array}{c}
        B_{1} \\
        \eps^{\scriptscriptstyle-1/2}B_{2} \\
      \end{array}
    \right)\in\R^{n\times k}, \quad \quad C^{\epsilon}=\left(
                             \begin{array}{c}
                               C_{1} \\
                               \eps^{\scriptscriptstyle-1/2}C_{2} \\
                             \end{array}
                           \right)\in\R^{n\times m}.$$
We assume that, for all $\eps>0$, the columns of $B^\eps$ lie in the column space of the matrix $C^\eps$, i.e.~$\mathrm{ran}(B^\eps)\subset\mathrm{ran}(C^\eps)$ or, equivalently, the column space of $B^\eps$ is orthogonal to the kernel of $(C^\eps)^T$, so that the equation
\begin{equation}
	C^\eps \xi = B^\eps c
\end{equation}
has a (not necessarily unique) solution for every $c\in\R^k$.
We seek a control $\alpha$ that minimises the following quadratic cost functional
           \begin{equation}\label{cost1}
           	J(\alpha; t, x)=\mathbb{E}_{t,x}\left[\frac{1}{2}\int_{t}^{\tau} R_s^{T}Q_{0}R_s+|\alpha_{s}|^{2}  ds+\frac{1}{2}R_\tau^T Q_{1}R_\tau\right],
           \end{equation}
       where $\tau$ is a bounded stopping time given by $\tau=\inf\{s \in [t,T]: X_t\notin S\}$ where $S$ is a bounded subset of $\R^{n_{s}}\times \R^{n_{f}}$  which containe the initial state $x$, and
           where $Q_{0}, Q_{1}\in\R^{n_{s}\times n_{s}}$ are any given symmetric positive semi-definite matrices. Note that even though the cost depends only on the slow process, the expected cost depends on the initial conditions of both $r$ and $u$. We call
           \begin{equation}\label{q0q1}
           	q_{0}= r^TQ_{0}r\,,\quad q_{1}=r^T Q_{1}r\,.
           \end{equation}
       	The corresponding value function is our QoI, it is given by
           \begin{equation}\label{value1}
           	V^{\eps}(t,x)=\displaystyle\inf_{\alpha\in \mathcal{A}}J(\alpha; t, x).
           \end{equation}
       where $\mathcal{A}$ is the space of all admissible controls $\alpha$, such that (\ref{xeps}) has a unique strong solution. (Likewise we may consider $q_0,q_1$ or $J$ to be our quantities of interest.)

           Assuming that all coefficients are known, the averaging principle for linear-quadratic control systems of the form (\ref{xeps})--(\ref{cost1}) implies that, under mild conditions on the system matrices, the value function $V^\eps$ converges uniformly on any compact subset of $[0,T]\times \R^n$ to a value function $v=v(t,r)$; see e.g.~\cite{KNH18}. The latter is the value function of the following linear-quadratic stochastic control problem: minimise the reduced cost functional
           \begin{equation}\label{redcost}
           	\bar{J}(\alpha; t,r)=\mathbb{E}_{t,r}\left[\frac{1}{2}\int_{t}^{\tau}q_{0}(\bar{R}_{s})+|\alpha_{s}|^{2}\,ds+\frac{1}{2}q_{1}(\bar{R}_{\tau})\right],
           \end{equation}
			subject to
           \begin{equation}\label{limiting1}
           	d\bar{R}_{s}=(\bar{A}\bar{R}_{s}+\bar{B}\alpha_s)ds+\bar{C}dW_{s},
           \end{equation}
           where the coefficients of the reduced system are given by
           \begin{equation}
           	\bar{A}=A_{11}-A_{12}A_{22}^{-1}A_{21}\,,\; \bar{B}=B_{1}-A_{12}A_{22}^{-1}B_{2}\,,\; \bar{C}=C_{1}-A_{12}A_{22}^{-1}C_{2}\,.
           \end{equation}

\subsection*{Multiscale system with unknown diffusion coefficient}

We suppose that the noise coefficients $C_1$ and/or $C_2$ are unknown. This situation is common in many applications, since especially the diffusion coefficient of the unresolved variables is difficult to estimate. Very often, however, an educated guess can be made as to which set or interval the unknown coefficient lies in. Specifically, we suppose that $(C_{1},\,C_{2})^T \in \mathcal{A}_{\scriptscriptstyle0, \infty}^{\Theta}$ which is the collection of all $\Theta$-valued adapted process on $[0, \infty)$ where $\Theta$ is a given bounded and closed subset in $\R^{(n_s+n_f)\times m}$.

Following the work by Denis and co-workers \cite{denis2011,denis2006} we exploit the link between the $G-$expectation framework and diffusion controlled processes and define
                   $$D^{\eps}\tilde{W}_{t}=\int_{0}^{t}C dW_{s},$$
for each $C^{\eps}\in\mathcal{A}_{\scriptscriptstyle0, \infty}^{\Theta}$, such that
                 \[
                 C^{\eps}=D^{\eps}\left(
                   \begin{array}{c}
                     C_{1} \\
                     C_{2} \\
                   \end{array}
                           \right)\,,\quad D^{\eps}=\left(
                           \begin{array}{cc}
                           	I_{n_{s}} & 0\\
                           	0& \eps^{-1/2}I_{n_{f}} \\
                           \end{array}
                           \right)
                           \]

so that $(\tilde{W}_{s})_{s\geq0}$ is a $d-$dimensional {G}-{B}rownian motion. As the main result, we will show below that the value function converges uniformly on any compact subset of $[0,T]\times\R^n$. The result does not rely on any compactness or periodicity assumptions of the fast variables with unknown diffusion; the key idea is to recast the fully nonlinear dynamic programming (or: G-Hamilton-Jacobi-Bellman) equation of the full G-stochastic optimal control problem as a G-FBSDE and then study convergence to the limiting G-FBSDE, which implies convergence of the corresponding dynamic programming equation.

\subsection*{Nonlinear dynamic programming equation} By the dynamic programming principle for controlled G-SDE \cite{fei2013optimal}, the G-Hamilton-Jacobi-Bellman (G-HJB) equation associated with our uncertain stochastic control problem (\ref{xeps})--(\ref{cost1}) reads
\begin{equation}\label{HJBred1}
    -\frac{\partial v^\eps}{\partial t}=\displaystyle\inf_{c}\{G(DD^{T}\colon \nabla^{2}v^\eps)+\langle\nabla v^\eps, A x+B c\rangle+\frac{1}{2}q_{0}+\frac{1}{2}|c|^{2})\},
\end{equation}
with terminal condition
\begin{equation}
	v^\eps(\tau,\cdot) = \frac{1}{2}q_1\,.
\end{equation}
Note that we $v^\eps$ is different from the value function $V^\eps$ in (\ref{value1}), since the diffusion coefficient in (\ref{value1}) is assumed constant, whereas, here, it is part of the nonlinear generator that involves a maximisation over the coefficient. Further note that we have dropped the $\eps$ in $A=A^\eps$, $B=B^\eps$ and $D=D^\eps$.
We can get rid of the outer infimum since the diffusion part is independent of the control variable, and
\[
\inf_c\left\{\langle\nabla V^\eps, B^\eps c\rangle + \frac{1}{2}|c|^{2}\right\} = -  \frac{1}{2}|c|^2_{BB^T}\,,
\]
where $|c^2|_{BB^T}=\langle c,BB^Tc\rangle$. This implies that (\ref{HJBred1}) is equivalent to
 \begin{equation}\label{HJBred2}
    \frac{\partial v}{\partial t}+G(DD^{T}\colon \nabla^{2}v)+\langle\nabla v, A x \rangle-\frac{1}{2}|c|^2_{BB^T} + \frac{1}{2}q_{0}=0\,,
\end{equation}
with the associated {G}-FBSDE system given by
\begin{equation}\label{I1}
    \begin{aligned}
           dX_{s}^{\eps}  = & A X_{s}^{\eps}ds+D\,d\tilde{W}_{s}\,,\; X_t^\eps=x \\
           Y^\eps_{t} = & \frac{1}{2}q_{1}(R_{\tau})-\frac{1}{2}\int_{t}^{\tau} q_0(R_s)\,ds + \frac{1}{2}\int_{t}^{\tau}|B^{T}(D^{T})^{\sharp}Z_s^{\eps}|^{2}\,ds \\
            & - \int_{t}^{\tau}Z_s^{\eps}d\tilde{W}_{s}-(K_{\tau}-K_{t})
         \end{aligned}
\end{equation}
Here
\begin{equation}
	Y_s^\eps = v^\eps(s,X^\eps_s)\,,\quad Z_s^\eps = \nabla v^\eps(s,X_s^\eps)\,,\quad t\le s\le \tau\,,
\end{equation}
and $\sharp$ denotes the Moore-Penrose pseudo inverse of a matrix. The process $K$ is a decreasing G-martingale with $K_0=0$ that is a consequence of the G-martingale representation theorem \cite{Peng07}.

\subsection*{Strong convergence of the quantity of interest}

Since the G-FBSDE is decoupled and running and terminal cost $q_0, q_1$ depend only on the resolved variables, we can infer the candidate for the limiting process:
\begin{equation}\label{limitin3}
	d\bar{R}_{s}=(\bar{A}\bar{R}_{s}+\bar{B}\alpha_s)\,ds+ \bar{D}\,d\tilde{W}_s
\end{equation}
with $\bar{D}\,d\tilde{W}$ given by
\begin{align*}
\bar{C}dW_{s} & =(C_{1}-A_{12}A_{22}^{-1}C_{2})dW_{s}\\ & = C_{1}dW_{s}-A_{12}A_{22}^{-1}C_{2}dW_{s} \\ & = d\tilde{W}_{s}-A_{12}A_{22}^{-1}d\tilde{W}_{s}\\
& =: \bar{D}\,d\tilde{W}\,.
\end{align*}
in other words, $\bar{D}=(I_{n_{s}}, -A_{12}A_{22}^{-1})$. The associated limiting G-FBSDE reads
\begin{equation}\label{I2}
	\begin{aligned}
		d\bar{R}_{s}  = & \bar{A}\bar{R}_{s}ds- \bar{D}\, d\tilde{W}_{s}\,,\; \bar{R}_t=r\\
		\bar{Y}_{s}  = & \frac{1}{2} q_{1}(\bar{R}_{\tau})-\frac{1}{2}\int_{t}^{\tau}q_0(\bar{R}_{s})\,ds + \frac{1}{2}\int_{t}^{\tau}|\bar{B}^{T}(\bar{D}^{T})^{\sharp}\bar{Z_s}|^{2}ds \\  & - \int_{t}^{\tau}\bar{Z_s}d\tilde{W}_{s} - (\bar{K}_{\tau}-\bar{K}_{t})
	\end{aligned}
\end{equation}
The corresponding limit G-HJB equation is then given by
\begin{equation}\label{HJBredf}
    \frac{\partial \bar{v}}{\partial t}+G(\bar{D}\bar{D}^{T}\colon\nabla^{2}\bar{v})+\langle\nabla \bar{v}, \bar{A} r \rangle-\frac{1}{2}|\nabla \bar{v}|^{2}_{\bar{B}\bar{B}^T} + \frac{1}{2} q_{0}=0.
\end{equation}
with the natural terminal condition
\begin{equation}
	\bar{v}(\tau,\cdot) = \frac{1}{2}q_1\,.
\end{equation}

\begin{theorem}\label{theo}
Let $v^{\eps}$ be the classical solution of the dynamic programming equation \ref{HJBred2} and $\bar{v}$ be the solution of \ref{HJBredf}, then, as $\eps\to 0$
\[
v^{\eps}\rightarrow \bar{v}\,,\quad \nabla v^{\eps}\rightarrow \nabla \bar{v}
\]
where the convergence of $v^\eps$ is uniform on any compact subset of $[0,T]\times \R^{n_s}$ and pointwise for $\nabla v^\eps$ for all $(t,x)\in [0,T]\times \R^{n_s}$.
\end{theorem}

\begin{proof}
Subtracting the G-BSDE part of (\ref{I2}) from (\ref{I1}) yields
\begin{equation}
	\begin{aligned}
    Y_{t}^\eps-\bar{Y}_{t} &= \frac{1}{2}q_{1}(R_\tau)-\frac{1}{2}q_{1}(\bar{R}_{\tau})                    -\frac{1}{2}\int_{t}^{\tau}q_0(R_{s})\,ds + \frac{1}{2}\int_{t}^{\tau}q_0(\bar{R}_{s})\,ds\\ & + \frac{1}{2}\int_{t}^{\tau}|B^{T}(D^{T})^{\sharp}Z_s^{\eps}|^{2}ds
 - \frac{1}{2}\int_{t}^{\tau}|\bar{B}^{T}(\bar{D}^{T})^{\sharp}\bar{Z_s}|^{2}ds\\
                      &-\int_{t}^{\tau}Z_s^{\eps}d\tilde{W}_{s}+\int_{t}^{\tau}\bar{Z_s}d\tilde{W}_{s}
                      -(K_{\tau}-K_{t})+(\bar{K}_{\tau}-\bar{K}_{t})
\end{aligned}
\end{equation}
Let $\gamma>0$ be arbitrary. Defining $y_{t}=Y_{t}^\eps- \bar{Y}_{t}, M_{t}=K_{t}-\bar{K}_{t}$, we can apply It\^{o}'s formula to $|y_{t}|^{2}e^{\gamma t}$ for $0\leq t <\tau\le T$, which yields

\begin{equation}\label{itoy2}
	\begin{aligned}
   |y_{t} |^{2}e^{\gamma t} & + \int_{t}^{\tau}|Z_s^{\eps}-\bar{Z_s}|^{2}d\langle   \tilde{W}\rangle_{s}+\int_{t}^{\tau}\gamma|y_{s}|^{2}e^{\gamma s}ds\\
    = &   \left|\frac{1}{2}q_{1}(R_\tau)-\frac{1}{2}q_{1}(\bar{R}_{\tau})\right|^{2}e^{\gamma \tau}  - \int_{t}^{\tau}y_{s}e^{\gamma s} \left(q_0(R_{s})-q_0(\bar{R}_{s})\right)ds\\
                 & + \int_{t}^{\tau}y_{s}e^{\gamma s}\left(|B^{T}(D^{T})^{\sharp}Z_s^{\eps}|^{2}-|\bar{B}^{T}(\bar{D}^{T})^{\sharp}\bar{Z_s}|^{2}\right)ds - (\bar{M}_{\tau}-\bar{M}_{t}),
\end{aligned}
\end{equation}
where
\[
\bar{M}_{\tau}-\bar{M}_{t}=2\int_{t}^{\tau}y_{s}e^{\gamma s}dM_{s}+2\int_{t}^{\tau}y_{s}e^{\frac{\gamma s}{2}}\left(Z_s^{\eps}-\bar{Z_{s}}\right)d\tilde{W}_{s}\,.
\]
It is convenient to write $e^{\gamma s}$ on the right hand side as $e^{\gamma s/2}e^{\gamma s/2}$. Now
dropping the quadratic variation term on the left and using Young's inequality (cf.~Lemma \ref{lem:young}) gives after rearranging terms
\begin{equation}
	\begin{aligned}
   |y_{t} |^{2}e^{\gamma t} + & \gamma\int_{t}^{\tau}|y_{s}|^{2}e^{\gamma s}ds+(\bar{M}_{\tau}-\bar{M}_{t}) \leq \left|\frac{1}{2}q_{1}(R_\tau)-\frac{1}{2}q_{1}(\bar{R}_{\tau})\right|^{2}e^{\gamma \tau}  \\
           &  +\int_{t}^{\tau}\left(\frac{\lambda_{1}}{2}|y_{s}|^{2}e^{\gamma s}+\frac{e^{\gamma s}}{2\lambda_{1}} \left(q_0(\bar{R}_{s}) - q_0(R_s)\right)^{2}\right)ds\\
                 &+\int_{t}^{\tau}\left(\frac{\lambda_{2}}{2}|y_{s}|^{2}e^{\gamma s}+\frac{e^{\gamma s}}{2\lambda_{2}}\left(|B^{T}(D^{T})^{\sharp}Z_s^{\eps}|^{2}-|\bar{B}^{T}(\bar{D}^{T})^{\sharp}\bar{Z_s}|^{2}\right)^{2}\right)ds,
\end{aligned}
\end{equation}
where we have defined $\lambda_1,\lambda_2$ by $\gamma=\lambda_1/2+\lambda_2/2$. As a consequence,
\begin{equation}
	\begin{aligned}
		|y_{t} |^{2}e^{\gamma t} +(\bar{M}_{\tau}-\bar{M}_{t}) & \leq \left|\frac{1}{2}q_{1}(R_\tau)-\frac{1}{2}q_{1}(\bar{R}_{\tau})\right|^{2}e^{\gamma \tau}  \\
		&  +\int_{t}^{\tau}\frac{e^{\gamma s}}{2\lambda_{1}} \left(q_0(\bar{R}_{s}) - q_0(R_s)\right)^{2}\,ds\\
		&+\int_{t}^{\tau} \frac{e^{\gamma s}}{2\lambda_{2}}\left(|B^{T}(D^{T})^{\sharp}Z_s^{\eps}|^{2}-|\bar{B}^{T}(\bar{D}^{T})^{\sharp}\bar{Z_s}|^{2}\right)^{2}\,ds.
	\end{aligned}
\end{equation}
Using the shorthands $N=(B_{1}, B_{2})^{T}$ and $k_{s}=\left(NZ_s^{\eps}+N\bar{Z}_s\right)$, with
$$\ds\left((B^{T}((D^{\eps})^{T})^{\sharp}Z_s^{\eps})-(\bar{B}^{T}(D^{T})^{\sharp}\bar{Z_s})\right)=\left(NZ_s^{\eps}-N\bar{Z_s}\right),$$
the pathwise convergence
\[\E\left[\sup_{t\in[0,T]}|R_t - \bar{R}_t|^2\right]=\mathcal{O}(\eps)
\]
as $\eps\to 0$ for any fixed diffusion coefficient (e.g. \cite{kifer2003,kifer2004}), together with the stability result of Zhang and Chen  \cite[Thm.~3.1]{zhang11}, then implies that
\begin{equation}
	\begin{aligned}
   |y_{t} |^{2}e^{\gamma t}+(\bar{M}_{\tau}-\bar{M}_{t}) \leq & \frac{l\eps^{2}e^{\gamma \tau}}{4}
                +\int_{t}^{\tau}\frac{l\eps^{2}e^{\gamma  s}}{2\lambda_{1}} ds\\
                 &+\int_{t}^{\tau}\frac{|k_{s}|^{2}\|NN^{T}\|_F}{2\lambda_{2}}|Z_s^{\eps}-\bar{Z_s}|^{2} e^{\gamma s}ds\\
	\end{aligned}
\end{equation}
for some generic constant $l\in(0,\infty)$ that may change from equation to equation. Taking the supremum and the using the fact that $\bar{M}$ is a symmetric G-martingale, it follows again by Young's inequality that

\begin{equation}\label{supy2}
   \Eh\left(\ds\sup_{s\in[t, \tau]}|y_{s} |^{2}e^{\gamma s}\right) \leq  \frac{l\eps^{2}e^{\gamma \tau}}{4} +\frac{l \eps^{2}e^{\gamma \tau}}{2\gamma\lambda_{1}}
                 +\frac{l}{2\lambda_{2}}\ds\Eh\left(\int_{t}^{\tau}|Z_s^{\eps}-\bar{Z_s}|^{2}e^{\gamma s}ds\right).
\end{equation}

Now using (\ref{itoy2}) again, together with the BDG-type inequalities (\ref{BDG1})--(\ref{BDG2}) for the quadratic variation and Young's inequality for the integrals involving $y_s e^{\gamma s}$ on the right hand side, we obtain after dropping the quadratic terms in $y$: 
\begin{equation}
	\begin{aligned}
   \su^{2}\Eh\left(\int_{t}^{\tau}|Z_s^{\eps}-\bar{Z_s}|^{2}e^{\gamma s}ds\right) \leq &   \frac{|l\eps|^{2}e^{\gamma \tau}}{4} +\int_{t}^{\tau}\frac{|l\eps|^{2}e^{\gamma s}}{2\alpha_{1}} ds\\
                 &+\frac{(k_{1})^{2}NN^{T}}{2\alpha_{2}}\ds\Eh\left(\int_{t}^{\tau}|Z_s^{\eps}-\bar{Z_s}|^{2}e^{\gamma s}ds\right).
	\end{aligned}
\end{equation}%
where $\alpha_1,\alpha_2$ are defined by $\gamma=\alpha_1/2+\alpha_2/2$. Hence 
\begin{equation}
	\begin{aligned}
   \Eh\left(\int_{t}^{\tau}|Z_s^{\eps}-\bar{Z_s}|^{2}e^{\gamma s}ds\right) \leq &  \frac{|l\eps|^{2}e^{\gamma \tau}}{4\su^{2}} +\int_{t}^{\tau}\frac{|l\eps|^{2}e^{\gamma s}}{2\alpha_{1}\su^{2}} ds\\
                 &+\frac{(k_{1})^{2}NN^{T}}{2\su^{2}\alpha_{2}}\ds\Eh\left(\int_{t}^{\tau}|Z_s^{\eps}-\bar{Z_s}|^{2}e^{\gamma s}ds\right),
	\end{aligned}
\end{equation}
which can be rearranged to give 
\begin{equation}
	\begin{aligned}
   \left(1-\frac{(k_{1})^{2}NN^{T}}{2\underline{l}\su^{2}\alpha_{2}}\right)\Eh\left(\int_{t}^{\tau}|Z_s^{\eps}-\bar{Z_s}|^{2}e^{\gamma s}ds\right) \leq  \frac{|l\eps|^{2}e^{\gamma \tau}}{4\underline{l}\su^{2}} +\frac{|l_{2}\eps|^{2}(e^{\gamma \tau}-e^{\gamma t})}{2\gamma\alpha_{1}\underline{l}\su^{2}}.
	\end{aligned}
\end{equation}

The last inequality can be combined with (\ref{supy2}), so that we obtain  
\begin{equation}
	\begin{aligned}
   \Eh\left(\ds\sup_{s\in[t, \tau]}|y_{s} |^{2}e^{\gamma s}\right) & +\left(1-\frac{(k_{1})^{2}NN^{T}}{2\underline{l}\su^{2}\alpha_{2}}-\frac{(k_{1})^{2}NN^{T}}{2\lambda_{2}}\right)\Eh\left(\int_{t}^{\tau}|Z_s^{\eps}-\bar{Z_s}|^{2}e^{\gamma s}ds\right)\\ \leq & \frac{|l\eps|^{2}e^{\gamma \tau}}{4}+\int_{t}^{\tau}\frac{|l_{2}\eps|^{2}(e^{\gamma \tau}-e^{\gamma t})}{2\gamma\lambda_{1}} +\frac{|l\eps|^{2}e^{\gamma \tau}}{4\underline{l}\su^{2}}
                  +\frac{|l_{2}\eps|^{2}(e^{\gamma \tau}-e^{\gamma t})}{2\gamma\alpha_{1}\underline{l}\su^{2}}.
	\end{aligned}
\end{equation}

As a consequence, 
\[\|Y^{\epsilon}-\bar{Y}\|_\gamma:=\Eh\left(\ds\sup_{s\in[t, \tau]}|y_{s} |^{2}e^{\gamma s}\right)\,,\quad \|Z^{\epsilon}-\bar{Z}\|_\gamma:=\Eh\left(\int_{t}^{\tau}|Z_s^{\eps}-\bar{Z_s}|^{2}e^{\gamma s}ds\right)
\]
go to zeros as $\epsilon\rightarrow 0$ at rate $\eps^2$. Since $\|\cdot\|_\gamma$ and $\|\cdot\|_{\gamma=0}$ are equivalent, it follows that $Y_t^{\epsilon}\rightarrow \bar{Y}_t$ uniformly for $t\in[0,T]$, and therefore, as $\eps\to 0$, 
\[
v^\eps(\cdot,x)=Y^\eps\to \bar{Y} =\bar{v}(\cdot,x)
\]
uniformly on any compact subset of $[0,T]\times \R^{n_s}$. Likewise, 
\[
\nabla v^\eps(t,x) = Z_t^{\epsilon}\to \bar{Z}_t=\nabla \bar{v}(t,x)\,, \quad (t,x)\in[0,T]\times \R^{n_s}
\]
as $\eps\to 0$, which implies the convergence of the optimal control in (\ref{xeps})--(\ref{cost1}).
\end{proof}

\begin{remark}
	The theorem also holds if the underlying G-SDE is nonlinear, as long as the averaging principle applies (e.g. when the drift is uniformly Lipschitz).
\end{remark}

\begin{remark}
	When $B=D$ in (\ref{HJBred1}) then the corresponding G-BSDE and the limit G-BSDE simplify to
	\begin{equation}
			Y^\eps_{t} = \frac{1}{2}q_{1}(R_{\tau})-\frac{1}{2}\int_{t}^{\tau} q_0(R_s)\,ds + \frac{1}{2}\int_{t}^{\tau}|Z_s^{\eps}|^{2}\,ds
			- \int_{t}^{\tau}Z_s^{\eps}d\tilde{W}_{s}-(K_{\tau}-K_{t})
	\end{equation}
and
\begin{equation}
		\bar{Y}_{s}  = \frac{1}{2} q_{1}(\bar{R}_{\tau})-\frac{1}{2}\int_{t}^{\tau}q_0(\bar{R}_{s})\,ds + \frac{1}{2}\int_{t}^{\tau}|\bar{Z_s}|^{2}ds - \int_{t}^{\tau}\bar{Z_s}d\tilde{W}_{s} - (\bar{K}_{\tau}-\bar{K}_{t})\,.
\end{equation}
\end{remark}

\section{Numerical illustration}\label{sec:num}

In this section we present two numerical examples to verify that the value function of the original system (\ref{HJBred2}) converges to the solution of the reduced system (\ref{HJBredf}) as $\eps\to 0$. The corresponding fully nonlinear HJB equations (\ref{HJBred2}) and (\ref{HJBredf}) are numerically solved by exploiting the link between fully nonlinear PDE and second-order BSDE (2BSDE); see e.g.  \cite{CSTV07}. The numerical algorithm for solving 2BSDE is based on the deep 2BSDE solver introduced by Beck et al. \cite{Beck2019}.

\subsection{Linear quadratic Gaussian regulator}

The first example is a 2-dimensional linear quadratic regulator problem given by the SDE
\begin{equation}
\label{controledeq2}
dX^{\epsilon}_t = (A^{\epsilon} X^{\epsilon}_t + B^{\epsilon} u^{\epsilon}_t) dt + \sqrt{\sigma} B^{\epsilon} dW_t, \\  X^{\epsilon}_0 = x_0,
\end{equation}
with unknown diffusion coefficient $\sigma \in [\underline{\sigma},\overline{\sigma}]$
and the cost functional
\begin{equation}
\label{bertfuctional2}
J(u;t,x) = \frac{1}{2} \E\left[\int_t^T ((X_s^{\epsilon})^T Q_0 X_s^{\epsilon}+|u^{\epsilon}_s|^2)ds+(X_T^{\epsilon})^T Q_1 X_T^{\epsilon}\right].
\end{equation}
Here $x=(r,u)\in\R^2$ and the coefficients are given by
\[
A^{\epsilon}=\left(
\begin{array}{cc}
	-2 & - 1/\epsilon  \\
	1/\epsilon  & -2 /\epsilon^{2}  \\
\end{array}
\right),\quad B^{\epsilon}=\left(
\begin{array}{c}
	0.1\\
	2/\epsilon  \\
\end{array}
\right), \quad Q_0=0\,,\quad
Q_1=\left(
\begin{array}{c}
	1 \\
	0  \\
\end{array}
\right),
\]
and we define the value function as $v^{\epsilon}(t,x)=\inf_{u}J(u;t,x)$. The G-PDE corresponding to the Gaussian regulator problem \eqref{bertfuctional2}--\eqref{controledeq2} is then given by
\begin{equation}\label{HJB2to1}
    \frac{\partial v^\epsilon}{\partial t}+G(a^{\eps}\colon \nabla^{2}v^\epsilon)+\langle\nabla v^\epsilon, A^\eps x \rangle-\frac{1}{2}{\left<B^\eps,z\right>}^2=0\,,\quad v^\epsilon(T,x)=x_1^2
\end{equation}
where we have used the shorthand $a^\eps=\sigma B^\eps(B^\eps)^T$. Calling $a=\sigma\bar{B}\bar{B}^T$, the G-PDE of the limiting value function $\bar{v}=\lim_{\eps\to 0} v^\eps$ then reads
\begin{equation}\label{HJB2to1lim}
    \frac{\partial \bar{v}}{\partial t}+G(a\colon \nabla^{2}\bar{v}) + \langle\nabla \bar{v}, \bar{A}\bar{x}\rangle-\frac{1}{2}\langle\nabla \bar{v},\bar{B}\rangle^{2}=0\,,\quad \bar{v}(T,r)=r\,.
\end{equation}
The function $G\colon \R\to  \R$ is defined by:
\begin{equation*}
   G(x) =\frac{x}{2}\left\{
         \begin{array}{ll}
           \bar{\sigma} \, \, \textsf{if} \, x\geq0 \\
          \underline{\sigma} \,\, \textsf{if} \, x<0\,.
         \end{array}
       \right.
\end{equation*}

\subsection*{Numerical results}

We consider the two value functions in the time interval $[0,0.1]$ with fixed initial condition $x=(r,u)=(1,0.5)$. For the diffusion coefficient, we assume $\sigma\in[0.8,1]$. The driver $f$ of the 2BSDE corresponding to the {G}-PDE \eqref{HJB2to1} of the original system is
\[
f(t,x,y,z,S)= G(a^\eps\colon S) + \frac{1}{2}{\left<B^\eps,z\right>}^2+\frac{1}{2}\left<A^\eps x,z\right>,
\]
whereas the 2BSDE corresponding to the limiting {G}-PDE \eqref{HJB2to1lim} has the driver
\[\bar{f}(t,x,y,z,\bar{S})= G(a\colon\bar{S})  +\frac{1}{2}{\left<\bar{B},z\right>}^2+\frac{1}{2}\left<\bar{A}x,z\right>.
\]
We compare $v^\epsilon (0,x)$ and $\bar{v}(0,r)$ and call
\[
\delta_v(\eps) = |v^\epsilon (0,x) - \bar{v}(0,r)|\,,
\]
Denoting by $u^\eps$ and $u$ the corresponding optimal controls for any given noise coefficient $\sigma$ (that can expressed in terms of the value function for fixed $\sigma$), we have
\[
v^\epsilon (0,x)=\Eh\left[\int_0^T |u_s^\epsilon|^2 ds\right]\,,\quad \bar{v}(0,r)=\Eh\left[\int_0^T |u_s|^2 ds\right]\,.
\]
The simulation results are shown in the following table:

\begin{tabular}{|p{3cm}|p{2cm}|p{2cm}|p{2cm}|}
 \hline
 $\epsilon$ & 0.3 & 0.2  & 0.1 \\
 \hline
 $\delta_v$  & 0.15  & 0.06  & 0.01  \\
\hline
\end{tabular}

\subsection{Triad system for climate prediction} We consider a stochastic climate model which can be represented as a bilinear system with additive noise  \cite{MTV99}
\begin{equation}
\label{2to3Exa}
dX^{\epsilon}(t) = \frac{1}{\epsilon^2}L(X^\epsilon(t))dt  + \frac{1}{\epsilon}B(X^{\epsilon}(t),X^\epsilon(t)) dt  + \frac{1}{\epsilon}\Sigma\,dW_t\,, \quad  X^{\epsilon}(0) = x,
\end{equation}
where $X^\epsilon(t)=(R_1^\epsilon(t), R_2^\epsilon(t), U^\epsilon(t))\in\R^3$ and
\[
 L(x) = -\begin{pmatrix}
	0 \\ 0 \\
	u
\end{pmatrix}, \quad
B(x,x) = \begin{pmatrix}
        A_1 r_2 u \\ A_2 r_1 u\\ A_3 r_1 r_2
      \end{pmatrix}, \quad
   \Sigma=\begin{pmatrix}
           0 \\ 0 \\ \lambda
          \end{pmatrix},
\]
where $0<\epsilon\ll 1$, and $A_1,A_2,A_3$ are real numbers such that
\[
A_1+A_2+A_3 = 0\,,
\]
and
\[
\lambda\in[\underline{\sigma},\overline{\sigma}]\,
\]
is the unknown diffusion coefficient. Equation (\ref{2to3Exa}), which is a time rescaled version of (\ref{homx})--(\ref{homy}), is a simplified stochastic turbulence model that comprises triad wave interactions between two climate variables $r_1,\,r_2$ and a single stochastic variable $u$.

\begin{figure}
	\begin{center}
		\includegraphics[width=0.695\textwidth]{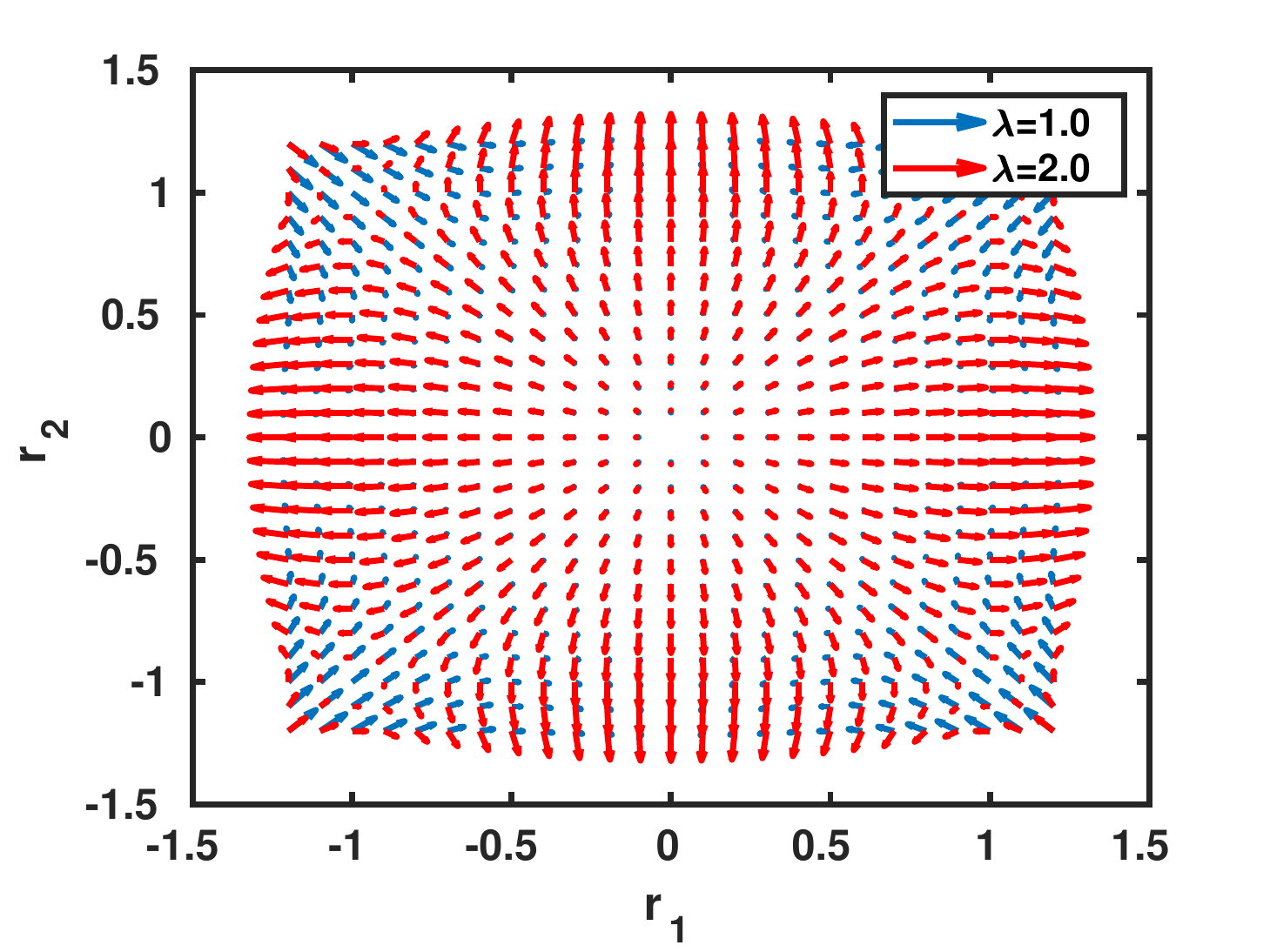}
		\caption{Vector field $f$ of the limit triad system for $A_1=A_2=1$ and $A_3=-2$ and two different noise parameters $\lambda$.}
		\label{fig:triad1}
	\end{center}
\end{figure}

The noise level $\lambda$  cannot be accurately estimated, nevertheless it may have a huge impact on the dynamics, even though there are no bifurcations for $\lambda>0$. Equation (\ref{2to3Exa}) can thus be considered an SDE driven by a G-Brownian motion. It is shown in \cite{MTV99} that, for any finite value $\lambda>0$, the first two components $R^\epsilon=(R_1^{\epsilon},R^\epsilon_2)$ converge strongly in $L^p$ for $p=1,2$ and on any bounded time interval $[0,T]$ to the solution of the nonlinear SDE with multiplicative noise
\begin{equation}
\label{2to3lim}
dR(t) = f(R(t)) dt + \sigma(R(t))  dW_t\,,\quad  R(0) = r,
\end{equation}
where $R(t) = (R_1(t), R_2(t))$ and
\[
f(r) =
      \begin{pmatrix}
        A_1 r_1 (A_3 r_2^2 +\frac{\lambda^2}{2}A_2) \\
         A_2 r_2 (A_3r_1^2+\frac{\lambda^2}{2}A_1)
      \end{pmatrix}, \quad
  \sigma(x)=\frac{\lambda}{\gamma}\begin{pmatrix}
  A_1r_2 \\ A_2r_1 \end{pmatrix}.
\]
The pathwise convergence $R^\eps\to R$ together with the stability result of Zhang and Chen  \cite[Thm.~3.1]{zhang11} implies that
\[\hat{\mathbb{E}}(\sup_{t\in[0, T]}|R^{\epsilon}(t)-R(t)|)\to 0\quad \textrm{as} \quad
\epsilon\to 0
\]
We can study the qualitative features of the triad system (\ref{2to3Exa}) in terms of the reduced model (\ref{2to3lim}). Using It\^o's formula, it readily follows that
\[
I(r_1,r_2) = A_1r_2^2 - A_2 r_1^2
\]
is a conserved quantity for both the reduced and the original system. We consider the case $A_1,\,A_2>0$ and $A_3<0$, in which case the level sets of $I$ are hyperbola, and the origin is an unstable hyperbolic equilibrium. The rays that connect the origin with any of the four equilibria
\[
r^*_{\pm,\pm} = \left(\pm\sigma\sqrt{\frac{A_1}{2|A_3|}},\,\pm\sigma\sqrt{\frac{A_2}{2|A_3|}}\right)\,,\quad A_1,\,A_2>0\,.
\]
are (locally hyperbolically unstable) invariant sets. Figure \ref{fig:triad1} shows representative vector fields $f$ of the limit system for different noise coefficients $\lambda=1.0$ and $\lambda=2.0$, when $A_1=A_2>0$. It can be seen that the repulsive and attractive regions on the invariant diagonals change as the coefficient $\lambda$ varies.

For illustration, Figure \ref{fig:triad2} shows three representative samples of $R(0.5)$ for $A_1=0.75$, $A_2=0.25$ and $A_3=-1.0$, with $\lambda=1.0$,  $\lambda=1.5$ and $\lambda=1.0$, all starting from the same initial value $R(0)=(1,-2)$. Note that the sample means over 100 independent realisations each depend on $\lambda$ in a non-trivial fashion.

   \begin{figure}
   	\begin{center}
   		\includegraphics[width=0.695\textwidth]{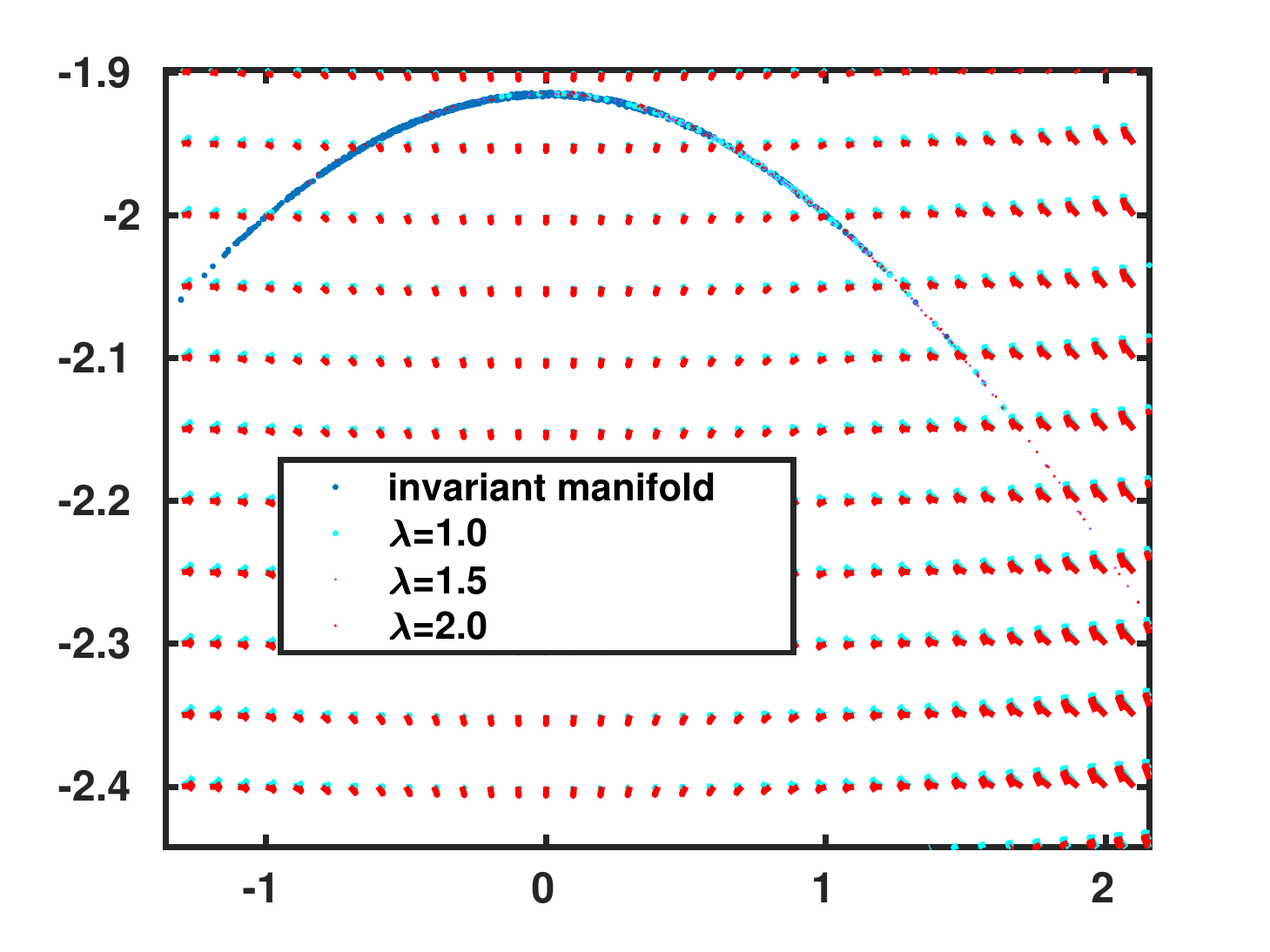}
   		\caption{Independent realisations of the limit triad system for $A_1=A_2=1$ and $A_3=-2$ and different noise parameters $\lambda\in[1,\,2]$ and fixed $T=0.5$. For every parameter value, we have generated 100 independent realisations, all starting from the same initial value $r=(1,-2)$. Note that the invariant manifolds, to which the trajectories are confined, are independent of $\lambda$, nevertheless the dynamics on the invariant manifolds are different.}
   		\label{fig:triad2}
   	\end{center}
   \end{figure}

\subsection*{Goal-oriented uncertainty quantification}

We now compare the full triad system (\ref{2to3Exa}) and the limit system (\ref{2to3lim}) for a specific quantity of interest (QoI) using the G-BSDE framework. To this end, we consider the QoI \emph{mean}
\begin{equation}
	v^\epsilon(t,x)=\mathbb{E}_{t,x}(X_1^\epsilon(T))\,,\quad v(t,r)=\mathbb{E}_{t,r}(R_1(T))
\end{equation}
as a function of the initial data $(t,x)$ and $(t,r)$ where  $x=(r,u)=(r_1,r_2,u)$ and $T>0$ is fixed. By definition,  the two value functions $v^\epsilon$
and $v$ solve the following nonlinear dynamic programming (HJB-type) equations
\begin{equation}\label{HJB2to3}
    \frac{\partial v^\epsilon}{\partial t} + G(a^{\eps}\colon\nabla^{2}v^\epsilon)+\langle\nabla v^\epsilon, b^\eps \rangle=0\,,\quad v^\epsilon(T,x)=x_1
\end{equation}
and
\begin{equation}\label{HJB2to3lim}
	\frac{\partial v}{\partial t}+G(a\colon \nabla^{2}v + \langle\nabla v, f_1 \rangle)+\langle\nabla v, f_2 \rangle = 0\,, \quad v(T,r)=r_1\,,
\end{equation}
with the shorthands
\[
b^\eps=\frac{1}{\eps^2}L+\frac{1}{\eps}B\,,\; a^\eps=\frac{1}{\eps^2}\Sigma\Sigma^T\,,\; a = \sigma\sigma^T\,, \; f_1 = \lambda^2A_1A_2 r\,,\; f_2=f-\frac{f_1}{2}\,.
\]
The nonlinearity $G$ in (\ref{HJB2to3}) and (\ref{HJB2to3lim}) is defined by
\begin{equation*}
   G(x) =\frac{x}{2}\left\{
         \begin{array}{ll}
           \overline{\sigma}  \, \, \textsf{if} \, x\geq0 \\
          \underline{\sigma}  \,\, \textsf{if} \, x<0
         \end{array}
       \right.
\end{equation*}
(We can think of $G$ as the nonlinear generator of the parameter-dependent part of the corresponding G-SDE.) We solve the fully nonlinear HJB equations by exploiting the aforementioned relation to second-order BSDE (2BSDE) and using the deep learning approximation developed by Beck et al. \cite{Beck2019}.

%
%

\subsection*{Numerical results}

As a first example, we consider the triad system and its homogenisation limit, with the parameters $ A_1=A_2=1,\,A_3=-2$ and $\lambda\in[0.8,\,1.2]$. Setting $T=0.1$ and $x=(r,u)=(1,-2,-2)^T$ the 2BSDE solution for $\epsilon=0.2$ yields the numerical approximations $v^\epsilon(0,x)=0.9291$ and $v(0,r)=0.9326$, i.e.
\[
\frac{|v^\epsilon(0,x)-v(0,r)|}{v(0,r)} = 0.0038
\]
in agreement with the theoretical prediction. We repeated the 2BSDE simulation for the same initial data and $\eps=0.2$, but with the different set of parameters $A_1=1,A_2=2,A_3=-3$, $\lambda\in[0.6,\,1.2]$ and $T=0.5$, and found $v^\epsilon(0,x)=1.3202$ and the limiting PDE $v(0,r)=1.3549$, i.e.
\[
\frac{|v^\epsilon(0,x)-v(0,r)|}{v(0,r)} = 0.0256\,.
\]

It is  illustrative to consider the parameter for which the maximum in the nonlinear part $G$ of the generator is attained. For example, for the original triad system,
\begin{equation}\label{Gmax}
G(a^{\eps}\colon\nabla^{2}v^\epsilon) = \max_{\lambda\in[\underline{\sigma},\overline{\sigma}]} a^{\eps}(\lambda)\colon\nabla^{2}v^\epsilon  = \frac{1}{\eps^2}\max_{\lambda\in[\underline{\sigma},\overline{\sigma}]} \lambda\frac{\partial^2 v^\epsilon}{\partial u^2}\,,
\end{equation}
which is identically equal to $\underline{\sigma}$ if $v^\eps$ is strictly concave in its third argument, $u$, and equal to $\overline{\sigma}$ if it is strictly convex in $u$. For a G-PDE of the form (\ref{HJB2to3}) that contains no running cost, one can show that the value function is strictly convex or concave if the terminal condition is strictly convex or concave (since the solution of the forward SDE is a strictly increasing function of the initial value).
In general, however, it is not the convexity that determines, for which parameter value the maximum is attained, as the limit G-PDE (\ref{HJB2to3lim}) shows. In fact, the optimal parameter will be a feedback function that depends on $(t,x)$ or $(t,r)$.

Figure \ref{fig:sigma} shows the maximiser in (\ref{Gmax}) as function of $t$ for a fixed value of $x$. It can be seen that the optimal parameter value is time-dependent, which underpins the fact that the optimal parameter depends on the QoI (here also through the initial data) in a nontrivial way; cf. Figure \ref{fig:triad2}.

\begin{figure}
	\begin{center}
		\includegraphics[width=0.695\textwidth]{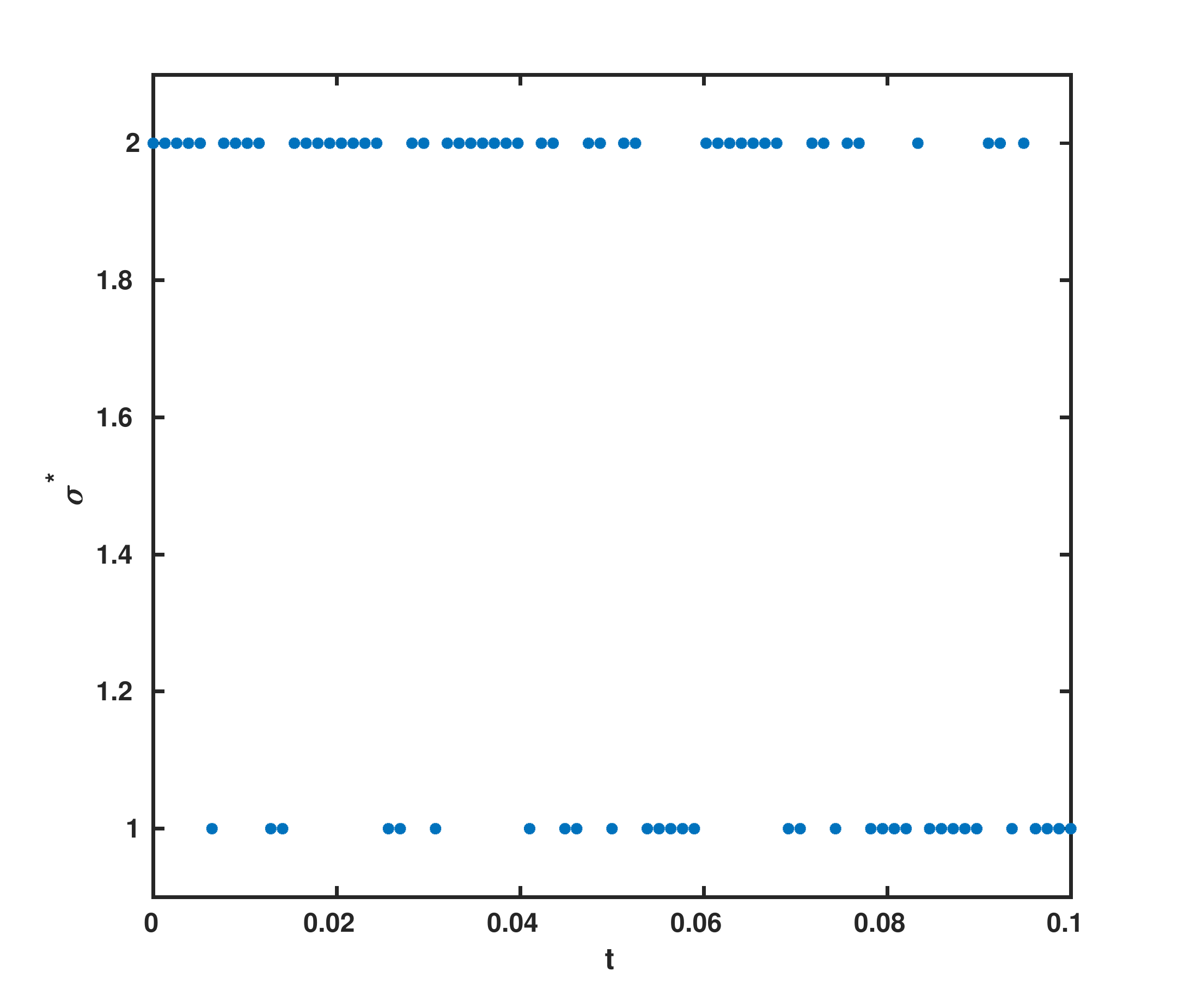}
		\caption{The plot shows the parameter $\sigma^*$ that maximises the nonlinear part $G$ of the generator in (\ref{Gmax}) for fixed initial condition over the noise coefficient $\lambda\in[1,2]$ .}
		\label{fig:sigma}
	\end{center}
\end{figure}

As a final numerical test, we consider the triad system with $A_1 = 0.75, A_2 = 0.25, A_3 = -1$ and $\lambda\in[1,\,2]$. For $T=0.1$ and $\eps=0.2$ we obtain $v^\epsilon(0,x)=0.9752$ and the limiting PDE $v(0,r)=0.9601$, i.e.
\[
\frac{|v^\epsilon(0,x)-v(0,r)|}{v(0,r)} = 0.0157\,.
\]

\section{Conclusions}\label{sec:sum}

We have sketched a general framework for goal-oriented uncertainty quantification and model reduction of parametric multiscale diffusions. The framework is based on the notion of sublinear G-expectations and the related G-Brownian motion. The sublinear expectation framework allows to define worst-case scenarios for any given, possibly path-dependent quantity of interest (QoI), and we have proved pathwise convergence of the corresponding G-BSDE for the case when the multiscale system depends on a small parameter that can be sent to zero.

Given the rather restrictive assumptions in this paper, it can only serve as a starting point for further studies. For example, we have assumed that the unknown parameters are from a compact set, and it would be desirable to allow for unbounded parameters. Since the nonlinear generator of the underlying G-Brownian motion may not be unambigously defined then, this calls for a suitable regularisation that is likely to have a Bayesian interpretation that may open up new algorithmic possibilities to quantify the uncertainty in the reduced system. Another obvious extension of this study is a formulation of the G-BSDE of the reduced system in a purely data-driven fashion using simulation data from the original model (rather than its value function). This will lead to a tracking-type functional for the QoI that enters the G-BSDE that determines the worst-case parameter(s) for the reduced model. Finally, in the combination with controlled systems, questions regarding the communtativity of the control optimisation with the nonlinear expectation remain to be addressed. All this questions will be adressed in forthcoming papers.

\section*{Acknowledgement}

This work has been partially supported by the Collaborative Research Center \emph{Scaling Cascades in Complex Systems} (DFG-SFB 1114) through project A05 and by the MATH+ Cluster of Excellence (DFG-EXC 2046) through the projects EP4-4 and EF4-6. Hafida Bouanani gratefully acknowledges funding by ATRST (Algeria).

\appendix

\section{Nonlinear expectation}\label{sec:prelim}

To begin with, we fix the notation and review the fundamentals of G-Brownian motion and the related nonlinear expectation \cite{geng14}.

Let $\Omega=C_0([0,\infty))$ denote the the space of real valued continuous functions  $(\omega_{t})_{t\ge 0}$ with the property $\omega_0=0$. We denote by $C_{ {\rm b,lip}}(\R^{d})$ the space of bounded and  Lipschitz continuous functions on $ \R^{d}$ and, for every $T>0$, we define
\[L_ {\rm ip}(\Omega_{T})=\left\{\varphi(B_{t_{1}}, \ldots, B_{t_{n}}): n\geq1,  t_{1}, \ldots, t_{n}\in[0, T],
\varphi\in C_{ {\rm b,lip}}( \R^{d\times n})\right\}
\]
and
\[
L_ {\rm ip}(\Omega) = \bigcup_{T=0}^\infty L_ {\rm ip}(\Omega_{T}).
\]
We call the corresponding Banach space
\[
L(\Omega)=(L_ {\rm ip}(\Omega),\|\cdot\|_{\infty})
\]
the Lipschitz space on $\Omega$ and define $L^p(\Omega)=\{X\in L(\Omega)\colon |X|^p\in L(\Omega)\}$.  Following Peng \cite{Peng07}, a sublinear expectation (also called: G-expectation) is a functional $\hat{\mathbb{E}}(\cdot)\colon  L_{\rm ip}(\Omega)\to \R$ with the following properties:
\begin{enumerate}
	\item $\hat{\mathbb{E}}(X)\geq\hat{\mathbb{E}}(Y)$ if $X\geq Y$ (\emph{monotonicity})
	\item $\hat{\mathbb{E}}(l)=l$ for every $l\in \R$ (\emph{preservation of constant})
	\item $\hat{\mathbb{E}}(X+Y)\leq\hat{\mathbb{E}}(X)+\hat{\mathbb{E}}(Y)$ (\emph{sub-additivity})
	\item $\hat{\mathbb{E}}(\lambda X)=\lambda\hat{\mathbb{E}}(X)$ for all $\lambda\geq 0$  (\emph{positive homogeneity}).
\end{enumerate}
The triple $(\Omega, L_{\rm ip}(\Omega), \hat{\mathbb{E}})$ is called a sublinear expectation space. The sublinear expectation admits a variational representation in terms of a family $\{\E_P\colon P\in \mathcal{P}\}$ of linear expectations where $\mathcal{P}$ is a family of probability measures on $(\Omega,\cB(\Omega))$:
\begin{equation}\label{Gexpec}
	\hat{\E}(X) = \max_{P\in\mathcal{P}}\E_P(X)\,,\quad X\in L_{\rm ip}(\Omega)
\end{equation}

The corresponding canonical process $(B_{t})_{t\geq 0}$ on the sublinear expectation space $(\Omega, L_{\rm ip}(\Omega), \hat{\mathbb{E}})$ is called a \textbf{G-Brownian motion} and is characterized as follows:
\begin{definition}
	A $d$-dimensional process $(B_{t})_{t\geq0}$ is called a \textit{G-Brownian motion} under the  sublinear expectation $\hat{\mathbb{E}}$ if the following properties hold:
	\begin{enumerate}
		\item $B_{0}(\omega)=0$.
		\item For every $t,s \geq 0$ and any $n\in \mathbb{N}$, the increment $B_{t+s}-B_{t}$ is independent of  of the collection $(B_{t_{1}}, B_{t_{2}}, \ldots, B_{t_{n}})$ of random variables,
		$0 \leq t_{1}\leq \ldots \leq t_{n} \leq t$ where two random vectors $X,Y$ are independent if for all bounded and Lipschitz continuous test functions $\varphi$
		\[
		\hat{\E}(\varphi(X,Y)) = \hat{\E}(\hat{\E}(\varphi(x,Y)|x=X)
		\]
		\item For every $t,s \geq 0$, the increment $B_{t+s}-B_{t}$ is normally distributed with mean 0 and covariance $s\Sigma$ where $\Sigma\in\R^{d\times d}$ is a symmetric and positive semidefinite matrix that is independent of $s$ or $t$.
	\end{enumerate}
\end{definition}

An interesting property of the G-Brownian motion is that its quadratic variation $\langle B\rangle$ has almost the same properties as the G-Brownian motion itself: $\langle B\rangle_0=0$, the increment $\langle B\rangle_{t+s} -\langle B\rangle_{t}$  is independent of the collection
\[
(\langle B\rangle_{t_{1}}, \langle B\rangle_{t_{2}}, \ldots, \langle B\rangle_{t_{n}})\,, \quad 0 \leq t_{1}\leq \ldots \leq t_{n} \leq t
\]
of quadratic variations, and $\langle B\rangle_{t+s} -\langle B\rangle_{t}\stackrel{d}{=}\langle B\rangle_{s}$.

\subsection{Path properties of G-Brownian motion}

We will need the following inequalities in the G-framework, all of which have straighforward interpretations in terms of the standard Brownian motion under the linear expectation. We define
\[
\mathit{M}^{p}([0,T])=\left\{\eta_{t}=\displaystyle\sum_{i=0}^{N-1}\xi_{i}\mathbbm{1}_{\{t_{i},
	t_{i+1}\}}: 0=t_{0}<\ldots<t_{N}=T, \xi_{i}\in L^p(\Omega_{t_{i}})\right\}
\]
to be the space of simple processes on $[0,T]$ and call $\mathit{M}_{G}^{p}([0,T])$ the completion of $\mathit{M}^{p}([0,T])$ under the norm
\[
\|\eta\|_{M,p}=\left(\hat{\mathbb{E}}\left[\int_{0}^{T}|\eta_{s}|^{p}ds\right]\right)^{\frac{1}{p}}
\]
and $\mathit{H}_{G}^{p}([0,T])$ the completion of $\mathit{M}^{p}([0,T])$ under the norm
\[
\|\eta\|_{H,p}=\left(\hat{\mathbb{E}}\left[\left(\int_{0}^{T}|\eta_{s}|^{2}ds\right)^{\frac{p}{2}}\right]\right)^{\frac{1}{p}}\,.
\]
In the following, let $B$ be a G-Brownian motion, with
\[
\underline{l}^2 = -\hat{\E}(-|B_1|^2)\le \hat{\E}(|B_1|^2)=\bar{l}^2 \,.
\]

\begin{proposition}\emph{(It\^{o} isometry inequality, \cite[Prop. 6.4]{Peng07})}\label{isometry} 
	Let $\beta\in M_{G}^{p}([0,T])$ for some $p\geq 2$. Then
	\[\int_{0}^{T}\beta_{t}dB_{t}\in L^{p}(\Omega_{T})
	\]
	and
	\begin{equation}\label{BDG1}
		\displaystyle\hat{\mathbb{E}}\left(\left|\int_{0}^{T}\beta_{t}dB_{t}\right|^{p}\right)\leq C_{p}\hat{\mathbb{E}}\left(\left|\int_{0}^{T}\beta_{t}^{2}d\langle B\rangle_{t}\right|^{\frac{p}{2}}\right),
	\end{equation}
	for some $C_p>0$.
	
\end{proposition}

\begin{proposition}\emph{(Burkholder-Davis-Gundy inequality, \cite[Prop. 2.6]{HJPS12})}\label{GBSDE} 
	For each $\eta \in H_G^{\alpha}([0,T])$ for some $\alpha\geq 1$ and $p\in(0,\alpha]$, we have
	\begin{equation}\label{BDG2}
		 \underline{l}^{p}c_{p}\displaystyle\hat{\mathbb{E}}\left[\left(\int_{0}^{T}\eta_{s}^{2}ds\right)^{\frac{p}{2}}\right]\leq\displaystyle\hat{\mathbb{E}}\left[\sup_{t\in[0, T]}\left|\int_{0}^{t}\eta_{s}dB_{s}\right|^{p}\right]\leq\bar{l}^{p}C_{p}\displaystyle\hat{\mathbb{E}}\left[\left(\int_{0}^{T}\eta_{s}^{2}ds\right)^{\frac{p}{2}}\right]
	\end{equation}
	where $0<c_{p}<C_{p}<\infty$ are constants.
	
\end{proposition}

\begin{proposition}\emph{(Isometry inequality, \cite[Lemma 2.19]{Lin})}\label{Bai}
	Let $p\geq1, \eta\in M_{G}^{p}([0, T])$ and $0\leq s\leq t\leq T$. Then
	\begin{equation}\label{E}
		\hat{\mathbb{E}}\left(\displaystyle\sup_{s\leq u\leq t}\left|\int_{s}^{u}\eta_{r}d\langle B\rangle_{r}\right|^{p}\right)\leq \left(\frac{\underline{l}+\bar{l}}{4}\right)^{p}(t-s)^{p-1}\hat{\mathbb{E}}\left(\int_{s}^{t}|\eta_{u}|^{p}du\right).
	\end{equation}
	
\end{proposition}
The previous inequalities in Propositions (\ref{BDG1}), (\ref{BDG2}) and (\ref{E}) hold true also for intervals in the form $[t,\tau].$

Now let
\[
S([0,T]):=\left\{h(t,B_{t_{1}\wedge t}, \ldots, B_{t_{n}\wedge t})\colon
t_{1}, \ldots, t_{n}\in[0, T], h\in C_{ {\rm b,lip}}( \R^{n+1}) \right\},
\]
with $S_{G}^{p}([0,T])$ denoting the completion of $S([0,T])$ under the norm
\[\|\eta\|_{
	S,p}=\left(\hat{\mathbb{E}}\left[\displaystyle\sup_{s\in[0,
	T]}|\eta_{s}|^{p}\right]\right)^{\frac{1}{p}}\,.
\]

\begin{corollary}
	For $\theta\in S_{G}^{2}$, we have
	\[\hat{\mathbb{E}}\left[\int_{0}^{T}|\theta_{s}|^{2}d\langle B\rangle_{s}\right]\leq T\bar{l}\hat{\mathbb{E}}\left[\displaystyle\sup_{s\in[0, T]}|\theta_{s}|^{2}\right]\,.
	\]
	Furthermore, for any $\eta\in H_G^{2}([0,T])$, the process
	\[
	\left(\int_{0}^{t}\eta_{s}\theta_{s}dB_s\right)_{t\in[0, T]}
	\]
	is a uniformly integrable martingale, with
	\[
	\hat{\mathbb{E}}\left[\int_{t}^{T}\eta_{s}\theta_{s}dB\right]=0\,.
	\]
	
\end{corollary}

\subsection{Some inqualities} We have

\begin{lemma}\label{lem:young}
	For $r>0$ and $1<q, p<\infty$, with $\frac{1}{p}+\frac{1}{q}=1$,  we have
	\begin{equation}\label{1.1}
		|a+b|^{r}\leq\max\{1, 2^{r-1}\}(|a|^{r}+|b|^{r})\quad\mbox{for}\quad a, b\in \R
	\end{equation}
	\begin{equation}\label{1.2}
		|ab|\leq\frac{|a|^{p}}{p}+\frac{|b|^{q}}{q}.
	\end{equation}
\end{lemma}

Using the sublinearity of $\hat{\E}$, the following is a straight consequence (see \cite{Peng07}):

\begin{proposition}
	For any $X,Y$ so that the moments below exist, we have
	\begin{equation}\label{1.3}
		\displaystyle\hat{\mathbb{E}}\left(|X+Y|^{r}\right)\leq 2^{r-1}\left(\hat{\mathbb{E}}(|X|^{r})+\hat{\mathbb{E}}(|Y|^{r})\right)
	\end{equation}
	\begin{equation}\label{1.4}
		 \hat{\mathbb{E}}(XY)\leq\displaystyle\left(\hat{\mathbb{E}}(|X|^{p})^{\frac{1}{p}}+\hat{\mathbb{E}}(|Y|^{q})^{\frac{1}{q}}\right)
	\end{equation}
	\begin{equation}\label{1.5}
		\left(\hat{\mathbb{E}}(|X+Y|^{p})\right)^{\frac{1}{p}}\leq \left(\hat{\mathbb{E}}(|X|^{p})\right)^{\frac{1}{p}}+\left(\hat{\mathbb{E}}(|Y|^{p})\right)^{\frac{1}{p}},
	\end{equation}
	where $r\geq 1$ and $1<p, q<\infty$, with $\frac{1}{p}+\frac{1}{q}=1$.
	In particular, for $1\leq p<p'$,
	\[
	\left(\hat{\mathbb{E}}(|X|^{p})\right)^{\frac{1}{p}}\leq\left(\hat{\mathbb{E}}(|X|^{p'})\right)^{\frac{1}{p'}}\,.
	\]
\end{proposition}

\section{2BSDE and fully nonlinear PDE}\label{sec:2BSDE}

We give the formal definition of a 2BSDE. For details we refer to \cite{CSTV07}.
\begin{definition}[2BSDE]
	Let $(t,x) \in [0,T) \times  \R^d$, $(X_s^{t,x})_{s \in [t,T]}$ a diffusion process and $(Y_s,Z_s, \Gamma_s, A_s)_{s \in [t,T]}$ a quadruple of $\mathbb{F}^{t,T}$-progressively measurable processes
	taking values in $ \R$, $ \R^d$, $\mathcal{S}^d$ and $ \R^d$, respectively. Then we say that the quadruple $(Y,Z,\Gamma,A)$ is a solution to the second order backward stochastic differential equation {\rm (2BSDE)} corresponding to $(X^{t,x},f,g)$ if
	\begin{eqnarray}
		\label{2bsde1}
		dY_s &=& f(s, X^{t,x}_s, Y_s ,Z_s , \Gamma_s) \,ds
		+ Z_s' \circ dX^{t,x}_s \, , \quad s \in [t,T) \, ,\\
		\label{2bsde2}
		dZ_s &=& A_s \,ds + \Gamma_s \,dX^{t,x}_s \, , \quad s \in [t,T) \, ,\\
		\label{2bsde3}
		Y_T &=& g\left(X^{t,x}_T\right) \, ,
	\end{eqnarray}
	where $Z_s' \circ dX^{t,x}_s$ denotes Fisk--Stratonovich
	integration, which is related to It\^{o} integration by
	$$
	Z_s' \circ dX^{t,x}_s =
	Z_s' \,dX^{t,x}_s + \frac{1}{2} \,d\left<{Z, X^{t,x}}_s\right>
	=
	Z_s' \,dX^{t,x}_s + \frac{1}{2} \,\mathrm{Tr}
	[\Gamma_s \sigma(X^{t,x}_s) \sigma(X^{t,x}_s)' ] \,ds \, .
	$$
\end{definition}
Now we present the relation between the 2BSDE (\ref{2bsde1})-(\ref{2bsde3}) and fully non-linear parabolic PDEs:
Let $f \colon [0,T) \times  \R^d \times\R\times\R^d \times \cS^d\to  \R$
and $g :  \R^d \to  \R$ are continuous functions, and assume $v : [0,T] \times  \R^d \to  \R$ is a $\mathcal{C}^{1,2}$ function such that
\[
v_t, Dv, D^2v, \cL D v \in \mathcal{C}^0([0,T) \times  \R^d)
\]
and $v$ solves the PDE
\begin{equation}\label{pde}
	- v_t(t,x) + f\left(t,x,v(t,x),Dv(t,x),D^2v(t,x)\right)
	= 0 \quad \mbox{on } [0,T) \times  \R^d \,,
\end{equation}
with terminal condition
\begin{equation} \label{terminal}
	v(T,x) = g(x) \, , \quad x \in  \R^d \,
\end{equation}
in the classical sense. Then it follows directly from It\^{o}'s formula that
for each pair $(t,x) \in [0,T) \times  \R^d$, the
processes
\begin{eqnarray*}
	Y_s &=& v\left(s,X^{t,x}_s\right) \, , \quad s \in [t,T] \, ,\\
	Z_s &=& Dv\left(s, X^{t,x}_s\right) \, , \quad s \in [t,T] \, ,\\
	\Gamma_s &=& D^2v\left(s, X^{t,x}_s\right) \, , \quad s \in [t,T] \, ,\\
	A_s &=& \cL Dv \left(s, X^{t,x}_s\right) \, , \quad s \in [t,T] \, ,
\end{eqnarray*}
solve the 2BSDE corresponding to $(X^{t,x}, f,g)$.

The converse is also true: The first component of the solution of the 2BSDE \eqref{2bsde1} at the initial time is a solution of the fully nonlinear PDE \eqref{pde}. We use this 2BSDE representation in Section \ref{sec:num} to solve fully nonlinear dynamic programming equations associated with a G-(B)SDE control problem.

\bibliography{gmms2}
\bibliographystyle{abbrv}

\end{document}